\documentclass[12pt,a4paper]{amsart}
\usepackage[latin1]{inputenc}
\usepackage{amssymb, amsmath}
\usepackage{geometry}
\usepackage{enumerate}
\usepackage[colorlinks=true,linkcolor=blue,urlcolor=blue,citecolor=blue]{hyperref}
\newtheorem{theorem}{Theorem}[section]
\newtheorem{definition}[theorem]{Definition}
\newtheorem{lemma}[theorem]{Lemma}
\newtheorem{corollary}[theorem]{Corollary}
\newtheorem{proposition}[theorem]{Proposition}
\newtheorem{remark}[theorem]{Remark}

\theoremstyle{definition}
\newtheorem{example}[theorem]{Example}

\begin{document}
\title[Rearrangement invariant optimal range for Hardy type operators]{Rearrangement invariant optimal range for Hardy type operators}
\author{Javier Soria}
\address{Department of  Applied Mathematics and Analysis, University of Barcelona, Gran Via 585, E-08007 Barcelona, Spain.}
\email{soria@ub.edu}

\author{Pedro Tradacete}
\address{Mathematics Department, University Carlos III de Madrid, E-28911 Legan\'es, Madrid, Spain.}
\email{ptradace@math.uc3m.es}

\thanks{Both authors have been partially supported by the Spanish Government Grant MTM2010-14946. The second author has also been partially supported by the Spanish Government Grant MTM2012-31286 and Grupo UCM 910346.}


\begin{abstract}
We characterize, in the context of rearrangement invariant spaces, the optimal range space for a class of monotone operators related to the Hardy operator. The connection  between optimal range and optimal domain for these operators is carefully analyzed.
\end{abstract}
\maketitle

\thispagestyle{empty}

\section{Introduction}

The main goal on this work is the study of optimal ranges for operators related to the Hardy operator
\begin{equation}\label{haop}
Sf(t)=\frac1t\int_0^tf(r)dr
\end{equation}
within the category of  rearrangement invariant spaces, r.i.\ for short  (see Definition~\ref{ridef}).

In recent years, the characterization of optimal domains, or range spaces, has been considered for many different kinds of operators and function spaces, with applications to Sobolev embeddings \cite{Edmunds-Kerman-Pick, CuRi}; classical results in harmonic analysis like the Hausdorff-Young inequality or Fourier multipliers \cite{Bennett,MoRi,MoRi2}; vector measures \cite{Okada-Ricker-Sanchez}, etc. We recall that if $T:X\rightarrow Y$ is bounded, the \textit{optimal domain}  is the largest space $Z$ (usually within a class of spaces with a priori conditions)  for which $T:Z\rightarrow Y$ is still bounded. Similarly, the \textit{optimal range} is the smallest space $Z$ such that $T:X\rightarrow Z$ is bounded. It is known that, in general, these spaces may not be well-defined. For example, for the particular case of the Hardy operator $S$, as in \eqref{haop}, it was proved in \cite{Delgado-Soria} that,  given any r.i. space, its optimal domain is never an r.i.\  space. In \cite{Nekvinda-Pick1} the optimal domain and range spaces for $L^p$ (among solid Banach spaces) were described. The same kind of results, but for optimal domains that are r.i., were considered in \cite{Delgado}.

Our aim in this work is to study the rearrangement invariant optimal range for some Hardy type operators, which will be defined in Section~\ref{sec claseH}. The main results will be proved in Section~\ref{sec2}; in fact, Theorems~\ref{optimal1} and \ref{optimal2} provide two different approaches in order to give explicit constructions of such spaces. After this more general discussion, in Section~\ref{sec3} our analysis will focus on a particular example: the Hardy operator $S$. In particular, special attention will be given to the r.i. optimal range for the Hardy operator defined on Lorentz and Marcinkiewicz spaces. 

Section~\ref{sec4} is devoted to the interesting relation between optimal range and optimal domain. Here we will study the behavior of iterating the constructions of optimal range and optimal domain. This will allow us to define two new functors on the category of r.i. spaces, namely $\mathcal{R}_X$ and $\mathcal{D}_X$ which correspond to the r.i. optimal range associated to the r.i. optimal domain for the Hardy operator on a space $X$, and vice versa. The basic properties of these functors will also be analyzed.

Our motivation  stems partly from several results on restricted type spaces given in  \cite{Soria-Tradacete}. In that paper, the r.i. optimal range for the Hardy operator on a Lorentz space $\Lambda_\varphi$ plays an important role in the characterization of the new class of spaces $R(X)$ (see in particular \cite[Corollary~2.4]{Soria-Tradacete}). In Section \ref{sec5}, as an application of the results given in previous sections, we will elaborate on the relation between the class of spaces $R(X)$ and the r.i. optimal range and domain. 
\medskip

\section{Terminology}\label{sec claseH}

In what follows, we will use the standard definitions and notations given in   \cite{Bennett-Sharpley}. Recall that given a measurable function $f$, $f^*$ denotes its decreasing rearrangement on $\mathbb{R}^+$.

\begin{definition}\label{ridef}
A rearrangement invariant Banach function space (r.i.) $X$ over a measure space $(\Omega,\Sigma,\mu)$ is the collection of all measurable functions $f:\Omega\rightarrow\mathbb{R}$ for which  $\|f\|_X<\infty$, where $\|\cdot\|_X$ satisfies the following properties:
\begin{itemize}
\item[(P1)] $\|\cdot\|_X$ is a norm;
\item[(P2)] $0\leq f^*\leq g^*\Rightarrow\|f\|_X\leq\|g\|_X$;
\item[(P3)] $0\leq f_n\uparrow f$ $\mu$-a.e. $\Rightarrow \|f_n\|_X\uparrow\|f\|_X$ (Fatou property);
\item[(P4)] $E\in\Sigma$, $\mu(E)<\infty\Rightarrow\|\chi_E\|_X<\infty$;
\item[(P5)] $\mu(E)<\infty\Rightarrow\int_E |f|d\mu\leq C_E\|f\|_X$, for some $C_E<\infty$ independent of $f$.
\end{itemize}

If a space $X$ satisfies all the above properties, except that Fatou property {\rm(P3)} is replaced by 
\begin{equation*}\label{P3*}\tag{P3$^*$}
f_n\in X,\  \sum_{n=1}^\infty\|f_n\|_X<\infty\Rightarrow\sum_{n=1}^\infty f_n=f\in X\hbox{ and } \Vert f\Vert_X\le\sum_{n=1}^\infty \Vert f_n\Vert_X<\infty,
\end{equation*}
then we say that $X$ is a rearrangement invariant Riesz-Fischer space.
\end{definition}

It is well known that \eqref{P3*}, which is weaker than (P3), holds for r.i.\ spaces \cite[Theorem I.1.6]{Bennett-Sharpley}, and that this condition implies completeness; i.e.,    they  are in fact Banach spaces. 

Given two measurable functions $f$ and $g$ we will write $f\prec g$ if for every $t>0$ it holds that
$$
\int_0^tf^{*}(s)\,ds\leq \int_0^tg^{*}(s)\,ds.
$$
This   is  known as the Hardy-Littlewood-Pólya relation.

If a rearrangement invariant Riesz-Fischer space also satisfies  that
\begin{equation*}\label{P2*}\tag{P2$^*$}
f\prec g\,\Rightarrow\|f\|_X \leq \|g\|_X,
\end{equation*}
then $X$ is called \emph{rearrangement invariant monotone Riesz-Fischer space}. Again, \eqref{P2*} always holds for r.i.\ spaces \cite[Theorem II.4.6]{Bennett-Sharpley}.

Given an r.i.\ space $X$, we denote by $X'$ its \textit{associate space,} defined as
$$
\Vert f\Vert_{X'}=\sup\bigg\{\int_{\Omega}|fg|\,d\mu:g\in X,\Vert g\Vert_X\le 1\bigg\}.
$$
Also, if $T:X\rightarrow Y$ is bounded, then we denote by $T':Y'\rightarrow X'$ the \textit{associate operator.}
The \textit{fundamental function} of $X$ is given by
$$
\varphi_X(t)=\Vert \chi_E\Vert_X,
$$
where  $\mu(E)=t$ and $ \chi_E$ denotes the characteristic function of the set $E$. The fundamental functions of $X$ and $X'$ satisfy that, for each $t>0$,
$$
\varphi_X(t)\varphi_{X'}(t)=t.
$$

As usual, we will use the notation $f^{**}(t)=Sf^*(t)=t^{-1}\int_0^t f^*(s)\,ds.$ For simplicity, we are going to work in the case $\Omega=\mathbb R^+$ and $d\mu$ will be the Lebesgue measure. The reader is referred to \cite{Bennett-Sharpley} for further notions concerning rearrangement invariant spaces.

\begin{definition}  $\varphi:\mathbb R^+\rightarrow\mathbb R^+$ is a quasiconcave function if $\varphi(t)$ is increasing and $\varphi(t)/t$ is decreasing.

Given a quasiconcave function $\varphi$, we define the Lorentz space
$\Lambda_\varphi$ as the class of those functions $f$ for which
$$
\Vert f\Vert_{\Lambda_\varphi}=\int_0^\infty f^*(t)\,d\varphi(t)<\infty.
$$

Similarly, the Marcinkiewicz space $M_\varphi$ is defined by the condition
$$
\Vert f\Vert_{M_\varphi}=\sup_{t>0} f^{**}(t) \varphi(t)<\infty.
$$
\end{definition}

It is well-known \cite[Theorem II.5.13]{Bennett-Sharpley} that if $X$ is an r.i.\ space  then, its fundamental function  $\varphi_X$ is quasiconcave, and
\begin{equation*}
\Lambda_{\varphi_X}\subset X\subset M_{\varphi_X}.
\end{equation*}

Motivated by   the Hardy operator \eqref{haop}, we are going  to specify some properties for operators between rearrangement invariant spaces, which will be needed to prove our main results.

\begin{definition}
Let $T:L^1\cap L^\infty\rightarrow L^1+L^\infty$ be an operator. 
\begin{enumerate}
\item[({\it i})] $T$ is decreasing if it maps positive decreasing functions to positive decreasing functions.
\item[({\it ii})] $T$ is HLP-monotone if $Tf \prec Tf^*$.
\item[({\it iii})] $T$ satisfies a restricted lower estimate, if $S\chi_{(0,t)}\leq ST\chi_{(0,t)}$, for any $t>0$.
\end{enumerate}
An operator $T$ is of class $\mathcal{H}$ whenever it satisfies (i), (ii) and (iii).

\end{definition}

Note that the Hardy operator $S$ is of class $\mathcal{H}$ (see \cite[Proposition II.3.2]{Bennett-Sharpley}). 

\begin{lemma}\label{associate operator}
Let $T$ be an operator such that $T\chi_{(0,t)}$ is positive decreasing for every $t>0$. Then, 
\begin{enumerate}
\item[({\it i})] $T'$ is HLP-monotone.
\item[({\it ii})] {If $T$ satisfies a restricted lower estimate, then for every decreasing function $g$ we have
$$
g\prec T'g.
$$
In particular, $T'$ also satisfies a restricted lower estimate.}
\end{enumerate}
\end{lemma}

\begin{proof}
\textit{(i)} For any $t>0$, and any positive function $f$, since $T\chi_{(0,t)}$ is decreasing, we have
\begin{align*}
\int_0^t(T'f)(s) \,ds&=\int_0^\infty f(s)(T\chi_{(0,t)})(s)\,ds\leq\int_0^\infty f^*(s)(T\chi_{(0,t)})^*(s)\,ds\\
&=\int_0^\infty f^*(s) (T\chi_{(0,t)})(s)\,ds=\int_0^t (T'f^*)(s) \,ds.
\end{align*}
Thus, $T'f \prec T'f^*$.

\textit{(ii)}  Suppose that for every $t>0$ we have
$$
\chi_{(0,t)} \prec T\chi_{(0,t)}.
$$
By \cite[Proposition II.3.6]{Bennett-Sharpley}, for every positive decreasing function $g$ on $\mathbb R^+$ we obtain
$$
\int_0^tg(s)\,ds=\int_0^\infty g(s)\chi_{(0,t)}(s)\,ds\leq\int_0^\infty g(s)(T\chi_{(0,t)})(s)\,ds=\int_0^t(T'g)(s)\,ds.
$$
Since this holds for every $t>0$, we get $g \prec T'g$.
\end{proof}

Let us see now some examples of operators within the class $\mathcal{H}$, apart from the Hardy operator.

\begin{example}
{\bf Hardy adjoint operator $S'$.} Let us recall the adjoint of the Hardy operator in $\mathbb{R}^+$:
\begin{equation*}
S'f(t)=\int_t^\infty\frac{f(s)}{s}ds.
\end{equation*}
Observe that for $s,t>0$
$$
S'\chi_{(0,t)}(s) = \chi_{(0,t)}(s)\log\Big(\frac{t}{s}\Big)=\log^+\Big(\frac{t}{s}\Big).
$$

Note that $S'$ is decreasing since $S'f$ is positive decreasing whenever $f$ is positive. By Lemma~\ref{associate operator} and the fact that $S$ is of class $\mathcal{H}$, we have that $S'$ is also HLP-monotone and satisfies a restricted lower estimate. Hence, $S'$ is of class $\mathcal{H}$. Observe that, in general, $S'f\not\le S'f^*$, for $f\ge0$ (take for instance $f=\chi_{(1,2)}$.)
\end{example}

\begin{example}
{\bf Iterations.} The operators $S^n(S')^m$ for every $n,m\in\mathbb{N}$ belong to the class $\mathcal{H}$. Indeed, since both $S$ and $S'$ are decreasing, so is the composition $S^n(S')^m$. Now, in order to see that $S^n(S')^m$ is HLP-monotone, when $m\geq1$, given $f$ we have
$$
(S^nS'^mf)^{**}=S(S^nS'^mf)^*=S^{n+1}S'^mf=S^nS'^mSf\leq S^nS'^mSf^*=(S^nS'^mf^*)^{**}.
$$
And analogously, the same estimate holds for $m=0$. 

Finally, by iterating the following facts: 
$$
SS'\chi_{(0,t)}(s)=S\chi_{(0,t)}(s)+S'\chi_{(0,t)}(s)\geq S\chi_{(0,t)}(s)
$$ 
$m$ times, and
$$
SS\chi_{(0,t)}(s)\geq S\chi_{(0,t)}(s)
$$
$n$ times, we get that $S^n(S')^m$ satisfies a restricted lower estimate. Hence $S^n(S')^m$ is of class $\mathcal{H}$.
\end{example}

\begin{example}\label{raoo}
{\bf Rank one operators.} If $w\in L^1_{\rm loc}(\mathbb R^+)$ is a positive decreasing weight, such that $w(r)\ge r^{-1/2}/2$, then it is easy to show that the operator
$$
T_wf(t)=w(t)\int_0^\infty f(s)w(s)\,ds,
$$
satisfies that $T_w:\Lambda_W\rightarrow L^1+L^\infty$, with $W(t)=\int_0^tw(r)dr$, and $T_w$ is of class $\mathcal{H}$.
\end{example}

\begin{lemma}
Let $T$ and $U$ be operators of class $\mathcal H$. For any positive scalars $\alpha$, $\beta$ such that $\alpha+\beta\geq1$, we have that $\alpha T+\beta U$ is of class $\mathcal{H}$.
\end{lemma}

\begin{proof}
Clearly, if $T$ and $U$ are decreasing, then so is $\alpha T+\beta U$. Now, given a function $f$, for $\alpha,\,\beta\geq0$, it follows that
\begin{align*}
\big((\alpha T+\beta U)f\big)^{**}&=\big(\alpha Tf+\beta Uf\big)^{**}\leq(\alpha Tf)^{**}+(\beta Uf)^{**}\\
&\leq(\alpha Tf^*)^{**}+(\beta Uf^*)^{**}=\big((\alpha T+\beta U)f^*\big)^{**},
\end{align*}
where the last identity follows from the fact that $\alpha Tf^*$ and $\beta Uf^*$ are positive decreasing functions. Thus, $\alpha T+\beta U$ is HLP-monotone. 

Finally, since $T$ and $U$ satisfy a restricted lower estimate, for any $t>0$ and   $\alpha+\beta\geq1$ we have
$$
S\chi_{(0,t)}\leq\alpha S\chi_{(0,t)}+\beta S\chi_{(0,t)}\leq \alpha ST\chi_{(0,t)}+\beta SU\chi_{(0,t)}=S(\alpha T+\beta U)\chi_{(0,t)}.
$$
Therefore, $\alpha T+\beta U$ also satisfies a restricted lower estimate, and is of class $\mathcal H$.
\end{proof}

\section{Rearrangement invariant optimal range for operators of class $\mathcal H$}\label{sec2}

In this Section we will describe the optimal range spaces, within the classes of r.i.\ Banach function spaces and also r.i.\  monotone Riesz-Fischer spaces. These, a priori, different classes have been previously considered in \cite{Bennett}, and appear as the optimal range for the Hausdorff-Young inequality (there are examples showing that, in fact, they may not be equivalent.)

\begin{definition}
Given an r.i.\  space $X$ and an operator $T:X\rightarrow L^1+L^\infty$, let us define the spaces
$$
\mathfrak{R}[T,X]_0=\Big\{f\in L^1+L^\infty:f^{**}\leq (Tg^*)^{**},\textrm{ for some  }g\in X\Big\},
$$
endowed with the norm
$$
\|f\|_{\mathfrak{R}[T,X]_0} = \inf\Big\{\|g\|_X: f^{**}\leq (Tg^*)^{**}\Big\},
$$
and
$$
\mathfrak{R}[T,X]=\bigg\{f\in L^1+L^\infty:\exists\lambda>0,\textrm{ such that }\forall g \downarrow \, \int_0^\infty f^*(t)g(t)\,dt\leq\lambda\|T' g\|_{X'}\bigg\},
$$
with the norm
$$
\|f\|_{\mathfrak{R}[T,X]}=\inf\bigg\{\lambda>0:\forall g\downarrow \,\int_0^\infty f^*(t)g(t)\,dt\leq\lambda\|T' g\|_{X'}\bigg\}=\sup_{g\downarrow}\frac{\int_0^\infty f^*(t)g(t)\,dt}{\|T' g\|_{X'}}.
$$
\end{definition}

\begin{theorem}\label{optimal1}
Let $X$ be an r.i.\  space and $T:X\rightarrow L^1+L^\infty$ be an operator of class $\mathcal{H}$. Then, the space $\mathfrak{R}[T,X]$ is the  r.i.\ optimal   range for the operator $T$ defined on $X$.
\end{theorem}

\begin{proof}

Let us check properties (P1)-(P5) for
$$
\|f\|_{\mathfrak{R}[T,X]}=\sup_{g\downarrow}\frac{\int_0^\infty f^*(s)g(s)\,ds}{\|T' g\|_{X'}}.
$$

(P1): Clearly, if $f=0$ then $\|f\|_{\mathfrak{R}[T,X]}=0$ and, for every scalar $\alpha$, we have $\|\alpha f\|_{\mathfrak{R}[T,X]}=|\alpha|\|f\|_{\mathfrak{R}[T,X]}$. Now, if $\|f\|_{\mathfrak{R}[T,X]}=0$, then for every $t>0$ we have that
$$
\int_0^tf^*(s)\,ds=\int_0^\infty f^*(s)\chi_{(0,t)}(s)\,ds\leq \|T'\chi_{(0,t)}\|_{X'}\|f\|_{\mathfrak{R}[T,X]}=0.
$$
This implies that $f=0$.

For the triangle inequality, let $f_1,f_2\in \mathfrak{R}[T,X]$. Observe that starting with the inequality $(f_1+f_2)^{**}\leq f_1^{**}+f_2^{**}$, by Hardy's Lemma \cite[Proposition II.3.6]{Bennett-Sharpley}, it follows that for any decreasing function $g$
\begin{align*}
\int_0^\infty(f_1+f_2)^*(s)g(s)\,ds&\leq\int_0^\infty (f_1^*+f_2^*)(s)g(s)\,ds\\&\leq\|f_1\|_{\mathfrak{R}[T,X]}\|T'g\|_{X'}+\|f_2\|_{\mathfrak{R}[T,X]}\|T'g\|_{X'}.
\end{align*}
That is
$$
\frac{\int_0^\infty(f_1+f_2)^*(s)g(s)\,ds}{\|T'g\|_{X'}}\leq \|f_1\|_{\mathfrak{R}[T,X]}+\|f_2\|_{\mathfrak{R}[T,X]},
$$
and taking the supremum over all decreasing functions $g$ we get the result. Thus, $\|\cdot\|_{\mathfrak{R}[T,X]}$ defines a norm.

(P2) is immediate. In order to check that Fatou property (P3) also holds, let $0\leq f_n\uparrow f$ a.e., so in particular, we also have $f_n^*\uparrow f^*$. Hence, by the monotone converge theorem, for any decreasing function $g$, we have that
$$
\int_0^\infty f_n^*(s) g(s)\,ds\uparrow\int_0^\infty f^*(s)g(s)\,ds.
$$
Therefore,
$$
\|f\|_{\mathfrak{R}[T,X]}=\sup_{g\downarrow}\frac{\int_0^\infty f^*(s)g(s)\,ds}{\|T' g\|_{X'}}=\sup_{g\downarrow}\frac{\sup_n\int_0^\infty f_n^*(s)g(s)\,ds}{\|T' g\|_{X'}}=\sup_n\|f_n\|_{\mathfrak{R}[T,X]}.
$$

Now, for property (P4), let $E\in\Sigma$ with $|E|<\infty$. Using that $S\chi_{(0,t)}\leq ST\chi_{(0,t)}$ for every $t>0$, together with H\"older's inequality, we get
\begin{align*}
\|\chi_E\|_{\mathfrak{R}[T,X]}& = \sup_{g\downarrow}\frac{\int_0^\infty \chi_{(0,|E|)}(s)g(s)\,ds}{\|T' g\|_{X'}} \leq \sup_{g\downarrow}\frac{\int_0^\infty T\chi_{(0,|E|)}(s)g(s)\,ds}{\|T' g\|_{X'}}\\
& = \sup_{g\downarrow}\frac{\int_0^\infty \chi_{(0,|E|)}(s)T'g(s)\,ds}{\|T' g\|_{X'}} \leq \|\chi_E\|_X<\infty.
\end{align*}

For the last property (P5), let $E\in \Sigma$ with $|E|<\infty$ and take a measurable function $f$. We have
\begin{align*}
\int_E f (x)\,dx&\leq \int_0^{|E|}f^*(s)\,ds=\|T'\chi_{(0,|E|)}\|_{X'}\frac{\int_0^\infty f^*(s)\chi_{(0,|E|)}(s)\,ds}{\|T'\chi_{(0,|E|)}\|_{X'}}\\
& \leq \|T'\chi_{(0,|E|)}\|_{X'}\|f\|_{\mathfrak{R}[T,X]},
\end{align*}
and $\|T'\chi_{(0,|E|)}\|<\infty$.

Therefore, $\mathfrak{R}[T,X]$ is an r.i.\  space. Now, let $f\in X$ (without loss of generality we may assume that $f\ge 0$). Since $(Tf)^{**}\leq(Tf^*)^{**}$ and $Tf^*$ is decreasing, we have that
\begin{align*}
\|Tf\|_{\mathfrak{R}[T,X]} &= \sup_{g\downarrow}\frac{\int_0^\infty (Tf)^*(s)g(s)\,ds}{\|T' g\|_{X'}} \leq \sup_{g\downarrow}\frac{\int_0^\infty Tf^*(s)g(s)\,ds}{\|T' g\|_{X'}}\\
& = \sup_{g\downarrow}\frac{\int_0^\infty f^*(s)T'g(s)\,ds}{\|T' g\|_{X'}} = \|f\|_X,
\end{align*}
so $T:X\rightarrow \mathfrak{R}[T,X]$ is bounded.

Now, suppose that $Y$ is another r.i.\  space such that $T:X\rightarrow Y$ is bounded, and let us see that  $\mathfrak{R}[T,X]\subset Y$. Since $T':Y'\rightarrow X'$ is bounded, then for any $g\in Y'$ we have
$$
\|T'g\|_{X'}\leq\|T'\|_{Y'\rightarrow X'}\|g\|_{Y'}=\|T\|_{X\rightarrow Y}\|g\|_{Y'}.
$$
Hence, for $f\in \mathfrak{R}[T,X]$, we have that
$$
\|f\|_Y = \sup_{g\downarrow}\frac{\int_0^\infty f^*(s)g(s)\,ds}{\|g\|_{Y'}} \leq \|T\|_{X\rightarrow Y}\sup_{g\downarrow}\frac{\int_0^\infty f^*(s)g(s)\,ds}{\|T' g\|_{X'}} = \|T\|_{X\rightarrow Y}\|f\|_{\mathfrak{R}[T,X]}.
$$
Therefore, we have that $\mathfrak{R}[T,X]\subset Y$ as claimed (note that the norm of this embedding is smaller than $\|T\|_{X\rightarrow Y}$). Thus, the space $\mathfrak{R}[T,X]$ is the r.i.\  optimal range for $T$ on $X$.
\end{proof}

\begin{theorem}\label{optimal2}
Let $X$ be an r.i.\  space and $T:X\rightarrow L^1+L^\infty$ be an operator of class $\mathcal{H}$. Then, the space $\mathfrak{R}[T,X]_0$ is the optimal rearrangement invariant monotone Riesz-Fischer range   for the operator $T$ defined on $X$.
\end{theorem}

\begin{proof}
Let us see that the space $\mathfrak{R}[T,X]_0$ is a rearrangement invariant monotone Riesz-Fischer space.

To this end, let us start by proving (P1); that is, $\|\cdot\|_{\mathfrak{R}[T,X]_0}$ actually defines a norm. We will prove the non-trivial part. Suppose that $\|f\|_{\mathfrak{R}[T,X]_0}=0$, then, there exists $g_n$ in $X$, with $f^{**}\leq~(Tg_n^*)^{**}$, such that $\|g_n\|_X\rightarrow0$. Since $T:X\rightarrow L^1+L^\infty$ is bounded, then $\|Tg_n^*\|_{L^1+L^\infty}\rightarrow~0$. From the condition $f^{**}\leq(Tg_n^*)^{**}$ for every $n$, we obtain     \cite[Theorem~II.4.6]{Bennett-Sharpley}
$$
\|f\|_{L^1+L^\infty}\leq \inf\|Tg_n^*\|_{L^1+L^\infty}=0,
$$
which shows that $f=0$.

For the triangle inequality, take  $f_1, f_2\in \mathfrak{R}[T,X]_0$. For each pair of   functions $g_1,g_2$ in $X$, such that $f_i^{**}\leq (Tg_i^*)^{**}$, for $i=1,2$, since $Tg_1^*$ and $Tg_2^*$ are decreasing functions, we have that

\begin{align*}
   (f_1+f_2)^{**}(t)&\leq  (Tg_1^*)^{**}(t)+(Tg_2^*)^{**}(t)=\frac1t\int_0^t\big((Tg_1^*)(s)+(Tg_2^*)(s)\big)\,ds \\
   &= \frac1t\int_0^t T(g_1^*+g_2^*)(s)\,ds=(T(g_1^*+g_2^*))^{**}(t).
\end{align*}

Therefore,
$$
\|f_1+f_2\|_{\mathfrak{R}[T,X]_0} \leq \|g_1^*+g_2^*\|_X \leq \|g_1\|_X+\|g_2\|_X,
$$
and since this holds for every $g_1,g_2\in X$ such that $f_i^{**}\leq (Tg_i^*)^{**}$, we get that $\|f_1+f_2\|_{\mathfrak{R}[T,X]_0} \leq \|f_1\|_{\mathfrak{R}[T,X]_0} + \|f_2\|_{\mathfrak{R}[T,X]_0}$.
\medskip

Property (P2$^*$) follows directly from the definition of the norm $\|\cdot\|_{\mathfrak{R}[T,X]_0}$. Let us see now that $\mathfrak{R}[T,X]_0$ satisfies the Riesz-Fischer property (P3$^*$).

To see this, let us suppose, without loss of generality, that $f_n$ is a sequence of positive functions in $X$ satisfying that  $\sum_{n=1}^\infty\|f_n\|_{\mathfrak{R}[T,X]_0}<\infty$. We want to prove that $\sum_{n=1}^\infty f_n$ converges in $\mathfrak{R}[T,X]_0$.

By hypothesis, for each $n$ let $g_n$ be a decreasing function in $X$ with
$$
f_n^{**}\leq (T g_n^*)^{**} \,\,\,\,\,\, \mathrm{ and }\,\,\,\,\,\,\|g_n\|_X\leq\|f_n\|_{\mathfrak{R}[T,X]_0}+2^{-n}.
$$
In particular,
$$
\sum_{n=1}^\infty \|g_n\|_X<\infty,
$$
and since $X$ is a Banach space, then the series $\sum_{n=1}^\infty g^*_n$ converges in $X$.

We claim that, for each $k>0$,
$$
\Big(\sum_{n=k}^\infty f_n\Big)^{**}\leq\Big(T\Big(\sum_{n=k}^\infty g^*_n\Big)\Big)^{**}.
$$
Indeed, since $T:X\rightarrow L^1+L^\infty$ is bounded, by  \cite[Theorem II.4.6]{Bennett-Sharpley},
$$
\sum_{n=1}^\infty \|f_n\|_{L^1+L^\infty}\leq\|T\|\sum_{n=1}^\infty\|g_n\|_X<\infty,
$$
and we get that $\sum_{n=1}^\infty f_n\in L^1+L^\infty$.

Now for fixed $k>0$, let
$$
h_{k,n}=\sum_{j=k}^{k+n}f_j.
$$
Clearly
$$
h_{k,n}\uparrow \sum_{j=k}^\infty f_j
$$
almost everywhere, so $h_{k,n}^{**}\uparrow (\sum_{j=k}^\infty f_j)^{**}$ pointwise. On the other hand, for each $n\in \mathbb{N}$, since $Tg^*_j\geq0$ and decreasing, we have
$$
h_{k,n}^{**}=\Big(\sum_{j=k}^{k+n} f_j\Big)^{**} \leq \sum_{j=k}^{k+n} (Tg^*_j)^{**} = \Big(T\Big(\sum_{j=k}^{k+n}g^*_j\Big)\Big)^{**} \leq \Big(T\Big(\sum_{j=k}^\infty g^*_j\Big)\Big)^{**}.
$$
Hence, taking the limit as $n\rightarrow\infty$ we have that
$$
\Big(\sum_{j=k}^\infty f_j\Big)^{**}\leq\Big(T\Big(\sum_{j=k}^\infty g^*_j\Big)\Big)^{**}
$$
as claimed.

Now, note that since
$$
\lim_{k\rightarrow\infty}\Big\|\sum_{n=k}^\infty g^*_n\Big\|_X=0 \hspace{1cm}\textrm{and}\hspace{1cm}\Big(\sum_{n=k}^\infty f_n\Big)^{**}\leq\Big(T\Big(\sum_{n=k}^\infty g_n^*\Big)\Big)^{**},
$$
then, we have that
$$
\lim_{k\rightarrow\infty}\Big\|\sum_{n=k}^\infty f_n\Big\|_{\mathfrak{R}[T,X]_0}=0,
$$
or equivalently, that $\sum_{n=1}^\infty f_n$ converges in $\mathfrak{R}[T,X]_0$. Therefore, (P3$^*$) holds.

In order to check property (P4), let $E\in\Sigma$ be such that $|E|<\infty$. Since $S\chi_{(0,|E|)} \leq ST\chi_{(0,|E|)}$, then $\chi^{**}_{(0,|E|)}\le(T\chi_{(0,|E|)})^{**}$, so we have that
$$
\|\chi_E\|_{\mathfrak{R}[T,X]_0} = \inf\{\|g\|_X:\chi_E^{**}\leq(Tg^*)^{**}\} \leq \|\chi_{(0,|E|)}\|_X<\infty.
$$

Finally for (P5), take $E\in\Sigma$, with $|E|<\infty$, and $f\in \mathfrak{R}[T,X]_0$. For every   $g\in X$ such that $f^{**}\leq(Tg^*)^{**}$ we have that
$$
\int_E f(x)\,dx\leq\int_0^{|E|}f^*(t)\,dt\leq\int_0^{|E|}Tg^*(t)\,dt\leq \|T'\chi_{(0,|E|)}\|_{X'}\|g\|_X.
$$
By hypothesis, since $T:X\rightarrow L^1+L^\infty$ is bounded, we have that $T':L^1\cap L^\infty\rightarrow X'$ is also bounded. In particular,  $T'\chi_{(0,|E|)}\in X'$ and since the above inequality holds for every $g$, we get
$$
\int_E f(x)\,dx \leq \|T'\chi_{(0,|E|)}\|_{X'}\|f\|_{\mathfrak{R}[T,X]_0}.
$$

So far, we have proved that $\mathfrak{R}[T,X]_0$ is a rearrangement invariant monotone Riesz-Fischer space. Moreover, for any $f\in X$,
$$
\|Tf\|_{\mathfrak{R}[T,X]_0} = \inf\big\{\|g\|_X:(Tf)^{**}\leq(Tg^*)^{**}\big\} \leq \|f^*\|_X.
$$
Thus, $T:X\rightarrow \mathfrak{R}[T,X]_0$ is bounded (with norm less than  or equal to one).

For the optimality, we consider any rearrangement invariant monotone Riesz-Fischer space $Y$ such that $T:X\rightarrow Y$ is bounded. If $f\in \mathfrak{R}[T,X]_0$ then,  for each  $g\in X$, with $f^{**}\leq (Tg^*)^{**}$, since $Y$ satisfies \eqref{P2*},  we have that
$$
\|f\|_Y \leq \|Tg^*\|_Y \leq \|T\| \|g\|_X.
$$
Now, taking the infimum over all such $g$, this yields that $\|f\|_Y \leq \|T\|\|f\|_{\mathfrak{R}[T,X]_0}$; i.e.,  $\mathfrak{R}[T,X]_0\subset Y$.
\end{proof}

\begin{remark}\label{comparison}{\rm Notice that we always have the embedding
$$
\mathfrak{R}[T,X]_0 \subseteq \mathfrak{R}[T,X],
$$
provided the two spaces are well defined. Indeed, let $f\in \mathfrak{R}[T,X]_0$ and pick any   function $g\in X$ such that $f^{**}\leq (Tg^*)^{**}$. Then, for every decreasing function $h$ we have that
$$
\int_0^\infty f^* (s)h(s)\,ds\leq \int_0^\infty (Tg^*)(s) h(s)\,ds\leq \|g\|_X \|T'h\|_{X'}.
$$
This implies that
$$
\|f\|_{\mathfrak{R}[T,X]}=\sup_{h\downarrow}\frac{\int_0^\infty f^*(s)h(s)\,ds}{\|T' h\|_{X'}} \leq \|g\|_X,
$$
and taking the infimum over all functions $g\in X$, with $f^{**}\leq (Tg^*)^{**}$, we get that $\|f\|_{\mathfrak{R}[T,X]}\leq \|f\|_{\mathfrak{R}[T,X]_0}$.}
\end{remark}

\begin{example}
Let us see what is the optimal range for the rank one operator $T_w$ of Example~\ref{raoo}. We recall that
$$
\Vert f\Vert_{(\Lambda_W)'}=\sup_{t>0}f^{**}(t)\frac{t}{W(t)}=\sup_{t>0}\frac{\int_0^tf^*(s)\,ds}{W(t)},
$$
where $W(t)=\int_0^tw(s)\,ds$ \cite[Theorem~2.4.7]{CRS}. Hence, using \cite[Theorem~3.2]{CaSo},
\begin{align*}
\Vert f\Vert_{\mathfrak{R}[T_w,\Lambda_W]}&=\sup_{h\downarrow}\frac{\int_0^\infty f^*(s)h(s)\,ds}{\Vert T'_wh\Vert_{(\Lambda_W)'}}=\sup_{h\downarrow}\frac{\int_0^\infty f^*(s)h(s)\,ds}{\Vert T_wh\Vert_{(\Lambda_W)'}}\\
&=\sup_{h\downarrow}\frac{\int_0^\infty f^*(s)h(s)\,ds}{\sup_{t>0}\frac{\int_0^tw(s)\int_0^\infty h(r)w(r)\,dr\,ds}{W(t)}}
=\sup_{h\downarrow}\frac{\int_0^\infty f^*(s)h(s)\,ds}{\int_0^\infty h(r)w(r)\,dr}\\
&=\sup_{a>0}\frac{\int_0^a f^*(s) \,ds}{\int_0^a  w(r)\,dr}=\Vert f\Vert_{M_{t/W(t)}}.
\end{align*}
Therefore $\mathfrak{R}[T_w,\Lambda_W]=M_{t/W(t)}$, with equality of norms.  Analogously,
\begin{align*}
\Vert f\Vert_{\mathfrak{R}[T_w,\Lambda_W]_0}&= \inf\bigg\{\Vert g\Vert_{\Lambda_W}:\int_0^tf^*(s)\,ds\le W(t)\int_0^\infty g^*(s)w(s)\,ds\bigg\}\\
&=\Vert f\Vert_{M_{t/W(t)}}.
\end{align*}

In particular, if we take $w(t)=1/\sqrt{t}$, then
$$
\mathfrak{R}[T_w,L^{2,1}]=\mathfrak{R}[T_w,L^{2,1}]_0=L^{2,\infty}.
$$
\end{example}

We now give a simple condition that characterizes when $\mathfrak{R}[T,X]_0=\mathfrak{R}[T,X]$. In Theorem~\ref{rangmar} we will prove that this is the case for the Hardy operator when $X$ is a Marcinkiewicz space (see also \cite{Bennett} for some related results).

\begin{corollary}
Let $X$ be an r.i.\  space and $T:X\rightarrow L^1+L^\infty$ be an operator of class $\mathcal{H}$. Then, the space $\mathfrak{R}[T,X]_0$ satisfies the Fatou property if and only if
$$
\mathfrak{R}[T,X]_0=\mathfrak{R}[T,X].
$$
\end{corollary}

\begin{proof}

Clearly, if $\mathfrak{R}[T,X]_0=\mathfrak{R}[T,X]$, then $\mathfrak{R}[T,X]_0$ has Fatou property. For the converse, notice that as was mentioned in Remark~\ref{comparison}, we always have
$$
\mathfrak{R}[T,X]_0\subseteq\mathfrak{R}[T,X].
$$
Since $\mathfrak{R}[T,X]_0$ is a rearrangement invariant Riesz-Fischer space, if it has also the Fatou property (P3), then it is an r.i.\  space. Since $T:X\rightarrow \mathfrak{R}[T,X]_0$ is bounded, by Theorem~\ref{optimal1} we also have
$$
\mathfrak{R}[T,X]\subseteq\mathfrak{R}[T,X]_0.
$$
\end{proof}

The existence of the r.i. optimal range of an operator $T$ of class $\mathcal H$ defined on an r.i. space $X$ is granted by the boundedness of $T:X\rightarrow L^1+L^\infty$. In fact, we have the following characterization:

\begin{proposition}\label{existence of range}
Let $T$ be an operator of class $\mathcal H$, and $X$ be an r.i.\  space. Then, the following are equivalent:
\begin{enumerate}[(i)]
\item There exists an  r.i.\ optimal  range for the operator $T$ on $X$.
\item $T:X\rightarrow L^1+L^\infty$ is bounded.
\item $T'\chi_{(0,t)}\in X'$, for every $t>0$.
\end{enumerate}
Moreover, if any of these conditions holds, then $\mathfrak{R}[T,X]$ is the  r.i.\ optimal  range for $T$ on $X$, and the norm of $T:X\rightarrow \mathfrak{R}[T,X]$ is always   equal to one.
\end{proposition}

\begin{proof}

$(i)\Rightarrow(ii)$ Let $Y$ be the r.i.\ optimal   space such that $T:X\rightarrow Y$ is bounded. Since the embedding $Y\hookrightarrow L^1+L^\infty$ is bounded, we have that
$$
T:X\rightarrow Y\hookrightarrow L^1+L^\infty
$$
is bounded.
\medskip

$(ii)\Rightarrow (iii)$ Since $T:X\rightarrow L^1+L^\infty$ is bounded, we also have that the associate operator $T':L^1\cap L^\infty\rightarrow X'$ is bounded. Now, since $\chi_{(0,t)}\in L^1\cap L^\infty$, for every $t>0$, it follows that $T'\chi_{(0,t)}\in X'$.
\medskip

$(iii)\Rightarrow(ii)$ Let $f\in X$.
\begin{align*}
\Vert Tf\Vert_{L^1+L^\infty}&=\int_0^1(Tf)^*(t)\,dt\le\int_0^1Tf^*(t)\,dt\\
&=\int_0^\infty f^*(t)T'\chi_{(0,1)}(t)\,dt\le\Vert f\Vert_X\Vert T'\chi_{(0,1)}\Vert_{X'}.
\end{align*}
Hence, since $T'\chi_{(0,1)}\in X'$, we have that $T:X\rightarrow L^1+L^\infty$ is bounded.

$(ii)\Rightarrow (i)$
As $T$ satisfies the hypothesis of Theorem~\ref{optimal1}, we get that $\mathfrak R[T,X]$ is the r.i. optimal range for $T$ on $X$.
\end{proof}

\begin{remark}{\rm
A direct application of  Proposition~\ref{existence of range} is the non existence of $\mathfrak{R}[S,L^1]$, since $S'\chi_{(0,t)}(s)=\log^+(t/s)\notin L^\infty$.
}
\end{remark}
Let us now have a look  at the fundamental function of the spaces $\mathfrak R[T,X]$. For this purpose, we introduce the following function:
\begin{equation}\label{psif}
\Psi_{T,X}(t)=\frac{t}{\|T'\chi_{(0,t)}\|_{X'}}.
\end{equation}

\begin{proposition}\label{fundamental function range}
Given an r.i. space $X$ and an operator $T:X\rightarrow L^1+L^\infty$ of class $\mathcal H$, the fundamental function of the space $\mathfrak R[T,X]$ satisfies that
$$
\varphi_{\mathfrak{R}[T,X]}(t)=\Psi_{T,X}(t).
$$
\end{proposition}

\begin{proof}
We observe that, for every $t>0$,
$$
\varphi_{\mathfrak{R}[T,X]}(t)=\sup_{g\downarrow}\frac{\int_0^tg(r)\,dr}{\|T'g\|_{X'}} \geq \frac{\int_0^t\chi_{(0,t)}(r)\,dr}{\|T'\chi_{(0,t)}\|_{X'}}=\Psi_{T,X}(t).
$$

For the converse inequality, given $t>0$, since $T$ is of class $\mathcal H$, note that
\begin{align}\label{normaT'chi}
\|T'\chi_{[0,t]}\|_{X'}&=\sup\bigg\{\int_0^\infty (T'\chi_{[0,t]})(r)\,f(r)\,dr: \|f\|_X\leq1\bigg\}\\
\nonumber &=\sup\bigg\{\int_0^tTf(r)\,dr:f \, \mathrm{decreasing,} \, \|f\|_X\leq1\bigg\}.
\end{align}
Now, for any such $f$, let us consider
$$
h_t(s)=\Big(\frac1t\int_0^tTf(r)\,dr\Big)\,\chi_{[0,t]}(s).
$$
By \cite[Theorem II.4.8]{Bennett-Sharpley}, we have that
$$
\Big(\frac1t\int_0^tTf (r)\,dr\Big)\,\varphi_{\mathfrak R[T,X]}(t)=\|h_t\|_{\mathfrak R[T,X]}\leq\|Tf\|_{\mathfrak R[T,X]}\leq\|f\|_X\leq 1,
$$
since the operator $T:X\rightarrow \mathfrak R[T,X]$ has norm 1. Hence, taking the supremum over $f$ and using \eqref{normaT'chi} we get
$$
\varphi_{\mathfrak R[T,X]}(t)\leq \frac{t}{\|T'\chi_{[0,t]}\|_{X'}}=\Psi_{T,X}(t).
$$

\end{proof}

\section{Optimal range for the Hardy operator}\label{sec3}

We are now going to study conditions on $X$ that guarantee the existence of the optimal range for the Hardy operator $S$. We will then consider particular cases in terms of the Lorentz  and Marcinkiewicz norms (see \cite{Soria-Tradacete} for some related results). Recall that for the Hardy adjoint operator we have
$$
S'\chi_{(0,t)}(s)=\chi_{(0,t)}\log\Big(\frac{t}{s}\Big).
$$

Now, let us summarize the results of the previous section for the particular case of the Hardy operator:

\begin{proposition}\label{prop41}
Let $X$ be an r.i.\  space. Then, the following are equivalent:
\begin{enumerate}[(i)]
\item There exists an  r.i.\ optimal  range for the Hardy operator $S$ on $X$.
\item $S:X\rightarrow L^1+L^\infty$ is bounded.
\item $\chi_{(0,1)}(t)\log\big({1}/{t}\big)=\log^+(1/t)\in X'$.
\item $S'\chi_{(0,t)}\in X'$, for every $t>0$.
\end{enumerate}
Moreover, if any of these conditions holds, then the  r.i.\ optimal range for Hardy operator on $X$ is given by
$$
\mathfrak{R}[S,X]=\bigg\{f\in L^1+L^\infty: \|f\|_{\mathfrak{R}[S,X]}=\sup_{g\downarrow}\frac{\int_0^\infty f(r) g(r)\,dr}{\|S'g\|_{X'}}<\infty\bigg\}.
$$ 
\end{proposition}

Proposition~\ref{fundamental function range}, yields that for any r.i. space $X$ such that $\log^+(1/t)\in X'$, the fundamental function of the space $\mathfrak R[S,X]$ is given by
$$
\varphi_{\mathfrak{R}[S,X]}(t)=\Psi_{S,X}(t)=\frac{t}{\|S'\chi_{(0,t)}\|_{X'}}.
$$

\begin{remark}{\rm
It is easy to see that 
$$
X\subset \mathfrak{R}[S,X]_0=\Big\{f\in L^1+L^\infty:f^{**}\leq (Sg^*)^{**},\textrm{ for some  }g\in X\Big\}.
$$
In fact, if $f\in X$, then taking $g=f^*\in X$, we get that $f^{**}=S(g)\le (Sg)^{**}$. In particular, if the upper Boyd index of $X$ (see \cite[Definition III.5.12]{Bennett-Sharpley}) satisfies that $\overline{\alpha}_X<1$ (which is equivalent  to the boundedness of  the operator $S:X\rightarrow X$), then
$$
X=\mathfrak{R}[S,X]_0=\mathfrak{R}[S,X].
$$
}
\end{remark}

We now study the optimal range for concrete examples of rearrangement invariant spaces.

\subsection{Lorentz spaces}

Note that Proposition~\ref{prop41}, yields that $S:\Lambda_\varphi\rightarrow L^1+L^\infty$ is bounded if and only if the function $\varphi$ satisfies that:
\begin{equation}\label{logc}
\varphi(t)\geq Ct\log(1+1/t).
\end{equation}
This fact was already proved in \cite{Soria-Tradacete}.

For an r.i.\ space $X$, it is clear that if $S:X\rightarrow L^1+L^\infty$ is bounded, then so is $S:\Lambda_{\varphi_X}\rightarrow L^1+L^\infty$ and, in particular, $\varphi_X(t)\geq Ct\log(1+1/t)$. However, the converse is not true for an arbitrary r.i.\  space. In fact, for the Marcinkiewicz space $X=M_{t\log(1+1/t)}$ it holds that the boundedness of $S:X\rightarrow L^1+L^\infty$ is given by the inequality:
$$
\int_0^1(Sf)^*(t)\,dt\le C\sup_{t>0}f^{**}(t)t\log(1+1/t),
$$
which, by  \cite[Theorem 6.4]{CPSS},   is equivalent to
$$
\int_0^1\frac{1}{t\log(1+1/t)}\,dt<\infty,
$$
and this is trivially false, since
$$
\int_0^1\frac{1}{t\log(1+1/t)}\,dt\approx\int_0^\infty\frac{du}{u}=\infty.
$$
Thus, $S:M_{t\log(1+1/t)}\rightarrow L^1+L^\infty$ is not bounded.
\medskip

We are now interested in estimating the fundamental function of the range space $\mathfrak{R}[S,\Lambda_\varphi]$. In \cite[Proposition 4.6]{Soria-Tradacete}, this was studied for   $\mathfrak{R}[S,\Lambda_\varphi]_0$, and  it was shown that the fundamental function satisfies that  
$$
\varphi_{\mathfrak{R}[S,\Lambda_\varphi]_0}(t)=\widetilde{\varphi}(t):=\inf_{r>0}\frac{t\varphi(r)}{r\log(1+\frac{t}{r})}.
$$
It follows that, essentially,  the same holds for $\mathfrak{R}[S,\Lambda_\varphi]$. 

\begin{proposition}\label{fundamental range lorentz}
Let  $\varphi$ be a quasiconcave function satisfying \eqref{logc}. Then
$$
\widetilde{\varphi}(t)=\varphi_{\mathfrak{R}[S,\Lambda_\varphi]_0}(t)\geq\varphi_{\mathfrak{R}[S,\Lambda_\varphi]}(t)\geq\Psi_{S,\Lambda_\varphi}(t)\geq\frac13\widetilde{\varphi}(t).
$$
\end{proposition}

\begin{proof}
As was mentioned before, the first identity was proved in \cite{Soria-Tradacete}. The embedding  $\mathfrak R[S,\Lambda_\varphi]_0\subset \mathfrak R[S,\Lambda_\varphi]$, implies the inequality $\varphi_{\mathfrak{R}[S,\Lambda_\varphi]_0}(t)\geq\varphi_{\mathfrak{R}[S,\Lambda_\varphi]}(t)$. We also have that
$$
\varphi_{\mathfrak{R}[S,\Lambda_\varphi]}(t)=\sup_{g\downarrow}\frac{\int_0^tg(r)\,dr}{\|S'g\|_{M_{\varphi'}}} \geq \frac{t}{\|S'\chi_{(0,t)}\|_{M_{\varphi'}}}=\Psi_{S,\Lambda_\varphi}(t).
$$

Finally, in order to show that $\Psi_{S,\Lambda_\varphi}(t)\geq\frac13\widetilde{\varphi}(t)$, note first that for $t\geq r$ we have
$$
1+\log\Big(\frac{t}{r}\Big)=\log\Big(\frac{et}{r}\Big)\leq\log\Big(\frac{3t}{r}\Big)\leq\log\Big(\Big(1+\frac{t}{r}\Big)^3\Big)=3\log\Big(1+\frac{t}{r}\Big).
$$

Now, for $t>0$ it holds that
\begin{align*}
\Psi_{S,\Lambda_\varphi}(t) & =  \frac{t}{\|S'\chi_{(0,t)}\|_{M'_\varphi}} =\inf_{r>0}\frac{\varphi(r)t}{rSS'\chi_{(0,t)}(r)}\\
 & = \min\Big\{\inf_{r\leq t}\frac{\varphi(r)t}{r(1+\log(\frac{t}{r}))},\inf_{r>t}\frac{\varphi(r)t}{r\frac{t}{r}}\Big\} =  \inf_{r\leq t}\frac{\varphi(r)t}{r(1+\log(\frac{t}{r}))}\geq\frac13\widetilde{\varphi}(t).
\end{align*}

\end{proof}

The following result characterizes those Lorentz spaces whose r.i. optimal range for the Hardy operator is again a Lorentz space.

\begin{proposition}\label{optimal range lorentz lorentz}
Let $\varphi$ be a quasi-concave function satisfying \eqref{logc}. The r.i. optimal range space is given by 
$$
\mathfrak{R}[S,\Lambda_\varphi]=\Lambda_{\widetilde\varphi},
$$
if and only if
$$
\int_t^\infty \frac{\widetilde{\varphi}(s)}{s^2}\,ds\lesssim \frac{\varphi(t)}{t}.
$$
\end{proposition}

\begin{proof}
By Proposition~\ref{fundamental range lorentz}, we have that $\varphi_{\mathfrak{R}[S,\Lambda_\varphi]}\approx \widetilde\varphi$. In particular, this implies that
$$
\Lambda_{\widetilde\varphi}\subset \mathfrak{R}[S,\Lambda_\varphi].
$$

Now, the converse embedding $\mathfrak{R}[S,\Lambda_\varphi]\subset\Lambda_{\widetilde\varphi}$ holds, if and only if the Hardy operator $S:\Lambda_\varphi\rightarrow \Lambda_{\widetilde\varphi}$ is bounded. This is in turn equivalent to the integral inequality
$$
\int_0^\infty f^{**}(s) d\widetilde\varphi(s) \lesssim \int_0^\infty f^*(s) d\varphi(s),
$$
and using \cite[Theorem 4.1]{CPSS} this holds if and only if
$$
\int_t^\infty \frac{\widetilde{\varphi}(s)}{s^2}\,ds\lesssim \frac{\varphi(t)}{t}.
$$
\end{proof}

\begin{example}\label{rango fi alfa}
For $\alpha\geq1$, let us consider the functions $\varphi_\alpha(t)=t\log^\alpha(1+t^{-1/\alpha})$. It holds that
$$
\mathfrak{R}[S,\Lambda_{\varphi_{\alpha+1}}]=\Lambda_{\varphi_\alpha}.
$$

Indeed, note first that if we denote by $\phi=\varphi_1$, then we have 
$$
\varphi_\alpha(t)=\phi(t^{1/\alpha})^\alpha.
$$
Using this identity, it can be easily seen that $\varphi_\alpha$ is a quasi-concave function satisfying $\varphi_\alpha(t)\geq\frac1\alpha t\log(1+1/t)$. Now, using that the function $\psi(t)=(1+t)\log(1+1/t)$ is decreasing for $t<1$, it can be shown that the fundamental function   satisfies
$$
\varphi_{\mathfrak{R}[S,\Lambda_{\varphi_{\alpha+1}}]}=\widetilde{\varphi_{\alpha+1}}=\varphi_{\alpha}.
$$

By Proposition~\ref{optimal range lorentz lorentz}, the statement reduces to proving the following estimate
$$
\int_t^\infty \frac{\varphi_\alpha(s)}{s^2}\,ds\lesssim \frac{\varphi_{\alpha+1}(t)}{t}.
$$

For $t<1$ we have
\begin{align*}
\int_t^\infty \frac{\varphi_\alpha(s)}{s^2}\,ds&=\int_t^\infty\frac1s\log^\alpha\Big(1+\frac1{s^{1/\alpha}}\Big)\,ds\approx \int_t^1\frac1s\log^\alpha\Big(1+\frac1{s^{1/\alpha}}\Big)\,ds + 1\\
&\leq\log^\alpha\Big(1+\frac1{t^{1/\alpha}}\Big)\log\Big(\frac1t\Big)+1\lesssim\log^{\alpha+1}\Big(1+\frac1{t^{1/\alpha}}\Big)\\
&=\frac{a+1}{a}\log^{\alpha+1}\Big(\Big(1+\frac1{t^{1/\alpha}}\Big)^{\frac\alpha{\alpha+1}}\Big)\lesssim \frac{\varphi_{\alpha+1}(t)}{t}.
\end{align*}
While for $t>1$ we have that 
\begin{align*}
\int_t^\infty \frac{\varphi_\alpha(s)}{s^2}\,ds&=\int_t^\infty\frac1s\log^\alpha\Big(1+\frac1{s^{1/\alpha}}\Big)\,ds\approx\int_t^\infty\frac{ds}{s^2}\\
&=\frac1t \approx \log^{\alpha+1}\Big(1+\frac1{t^{\alpha+1}}\Big)=\frac{\varphi_{\alpha+1}(t)}{t}.
\end{align*}
\end{example}

\subsection{Marcinkiewicz spaces} We study now the  r.i.\  optimal range for the Hardy operator acting on a Marcinkiewicz space $M_\varphi$. Let us start by computing the function $\Psi_{S,M_\varphi}$ defined  in \eqref{psif}:

\begin{lemma}\label{Psi Marcinkiewicz}
Let $\varphi$ be a quasiconcave function such that $ 1/\varphi$ is locally integrable at zero. Then, for $t>0$:
\begin{equation}\label{epmf}
\Psi_{S,M_\varphi}(t)=t\Big(\int_0^t\frac {ds}{\varphi(s)}\Big)^{-1}.
\end{equation}
\end{lemma}

\begin{proof}
Indeed, for $s>0$ let us consider the function
$$
f_s(t)=S'\chi_{(0,s)}(t)=\chi_{(0,s)}(t)\log\Big(\frac{s}{t}\Big),
$$
whose distribution function is given by
$$
\lambda_{f_s}(r)=|\{t>0:f_s(t)>r\}|=se^{-r}.
$$

Hence, for $t>0$ we have
\begin{align*}
\Psi_{S,M_\varphi}(t) & = \frac{t}{\|S'\chi_{(0,t)}\|_{M'_\varphi}} = \frac{t}{\int_0^\infty \frac{te^{-r}}{\varphi(te^{-r})}dr}\\
 & =  \Big(\int_0^\infty\frac{dr}{e^r\varphi(te^{-r})}\Big)^{-1} =  t\Big(\int_0^t\frac {ds}{\varphi(s)}\Big)^{-1}.
\end{align*}
\end{proof}

Recall that if $ 1/\varphi$ is locally integrable, then the space $M_{t\left(\int_0^t\frac{ds}{\varphi(s)}\right)^{-1}}$ is the minimal r.i.\  space containing $\Lambda_\varphi^{1,\infty}$ \cite[Theorem 3.3]{Rodriguez-Soria}, where
$$
\Lambda_\varphi^{1,\infty}=\Big\{f:\sup_{t>0} f^*(t)\varphi(t)<\infty\Big\}.
$$
Note that in general the space $\Lambda_\varphi^{1,\infty}$ need not be normable, except when $\varphi$ belongs to the  Ari\~no and Muckenhoupt $B_1$ class \cite{ArMu,So3}. This allows us to prove the following:

\begin{theorem} \label{rangmar}
Let $\varphi$ be a quasiconcave function. Then,  $S:M_\varphi\rightarrow L^1+L^\infty$ is bounded if and only if $ 1/\varphi$ is locally integrable at zero. Moreover, in this case we have
$$
\mathfrak{R}[S,M_\varphi]=\mathfrak{R}[S,M_\varphi]_0=M_{\Psi_{S,M_\varphi}}.
$$
\end{theorem}

\begin{proof}
Proposition~\ref{prop41} gives us that $S:M_\varphi\rightarrow L^1+L^\infty$ is bounded, if and only if $\chi_{(0,1)}(t)\log\big({1}/{t}\big)\in (M_\varphi)'=\Lambda_{t/\varphi(t)}$. But,
$$
\int_0^1\log(1/t)\,d(t/\varphi(t))<\infty \iff \int_0^1\frac{dt}{\varphi(t)}<\infty.
$$
By Lemma~\ref{Psi Marcinkiewicz}, we have that
$$
\Psi_{S,M_\varphi}(t)=t\Big(\int_0^t\frac{ds}{\varphi(s)}\Big)^{-1}.
$$

Now, it is clear that a decreasing function $f$ belongs to $M_\varphi$ if and only if $Sf$ belongs to $\Lambda_{\varphi}^{1,\infty}$, since
$$
\|Sf\|_{\Lambda_\varphi^{1,\infty}} = \sup_{t>0} Sf(t)\varphi(t) = \|f\|_{M_\varphi}.
$$
In particular, we have that
$$
S:M_\varphi\rightarrow\Lambda_\varphi^{1,\infty}\hookrightarrow M_{\Psi_{S,M_\varphi}}
$$
is bounded. Since Theorem~\ref{optimal1} yields that $\mathfrak{R}[S,M_\varphi]$ is the r.i.\  optimal range, we must have
$$
\mathfrak{R}[S,M_\varphi]\subseteq M_{\Psi_{S,M_\varphi}}.
$$

Conversely, since $ {t}/{\varphi(t)}$ is the associate function to $\varphi$, in particular it is quasiconcave, and there exists a positive decreasing function $h$ such that
$$
\frac{t}{\varphi(t)}\approx\int_0^t h(s)\,ds.
$$
Observe that the function $h$ satisfies that
$$
\|h\|_{M_\varphi}=\sup_{t>0} Sh(t)\varphi(t)\approx\sup_{t>0}\frac{1}{\varphi(t)}\varphi(t)=1.
$$
Now, given $f\in \Lambda_\varphi^{1,\infty}$ we have that, for any decreasing function $g$
\begin{align*}
\int_0^\infty f^*(t)g(t)\,dt & \leq  \|f\|_{\Lambda_\varphi^{1,\infty}}\int_0^\infty\frac{g(t)}{\varphi(t)}\,dt \approx \|f\|_{\Lambda_\varphi^{1,\infty}}\int_0^\infty\frac{g(t)}{t}\int_0^t h(s)\,ds \,dt\\
 & =  \|f\|_{\Lambda_\varphi^{1,\infty}}\int_0^\infty S'g(t) h(t)\,dt\leq \|f\|_{\Lambda_\varphi^{1,\infty}}\|S'g\|_{M'_\varphi}\|h\|_{M_\varphi}.
\end{align*}
Therefore, we have that $f\in \mathfrak{R}[S,M_\varphi]$ with
$$
\|f\|_{\mathfrak{R}[S,M_\varphi]}=\sup_{g\downarrow}\frac{\int f^*(s)g(s)\,ds}{\|S'g\|_{M'_\varphi}}\lesssim\|f\|_{\Lambda_\varphi^{1,\infty}}.
$$
Now, since $\mathfrak{R}[S,M_\varphi]$ is an r.i.\  space satisfying $\Lambda_\varphi^{1,\infty}\subseteq\mathfrak{R}[S,M_\varphi]$, then \cite[Theorem~3.3]{Rodriguez-Soria}
$$
M_{\Psi_{S,M_\varphi}}\subseteq\mathfrak{R}[S,M_\varphi].
$$

So far we know that
$$
\mathfrak{R}[S,M_\varphi]_0\subseteq\mathfrak{R}[S,M_\varphi]=M_{\Psi_{S,M_\varphi}}.
$$
For the remaining  embedding, let $f\in M_{\Psi_{S,M_\varphi}}$. Then, for every $t>0$ we have
\begin{align*}
(Sf)^*(t) &\leq \|f\|_{M_{\Psi_{S,M_\varphi}}} \frac{1}{t}\int_0^t\frac{ds}{\varphi(s)} \approx \|f\|_{M_{\Psi_{S,M_\varphi}}} \frac{1}{t}\int_0^t\frac1s\int_0^s h(u)du\\
& =\|f\|_{M_{\Psi_{S,M_\varphi}}} SS(h)(t).
\end{align*}
Hence, $f\in \mathfrak{R}[S,M_\varphi]_0$, with
$$
\|f\|_{\mathfrak{R}[S,M_\varphi]_0}\leq\|f\|_{M_{\Psi_{S,M_\varphi}}}\|h\|_{M_\varphi}\lesssim \|f\|_{M_{\Psi_{S,M_\varphi}}}.
$$
\end{proof}

\begin{example}\label{exuii}{\rm Since $\varphi(t)=\max\{1,t\}$ satisfies the hypothesis of Theorem~\ref{rangmar} and, using \eqref{epmf}, we have that
$$
\Psi_{L^1\cap L^\infty}(t)\approx\frac{t}{\log(1+t)},
$$
then  $\mathfrak{R}[S,L^1\cap L^\infty]=\mathfrak{R}[S,L^1\cap L^\infty]_0=M_{\frac{t}{\log(1+t)}}$.
}
\end{example}

\begin{example}
As before, we have
$$
\mathfrak{R}[S,M_{\frac{t}{\log(1+t)}}]=M_{\frac{t}{\log^2(1+t^{1/2})}}.
$$
\end{example}

\section{Optimal domain vs. optimal range}\label{sec4}

In this section we focus on the relation between r.i.\ optimal range and optimal domain for the Hardy operator. Some of these results can also be extended to operators of class $\mathcal H$.

Let us start with the characterization of the r.i. optimal domain (and its existence) for the Hardy operator $S$.

\begin{theorem}
Given an r.i.\ space $X$, the following are equivalent:
\begin{enumerate}[{(i)}]
\item There exists the r.i.\ optimal domain for the Hardy operator into the space $X$.
\item $S:L^1\cap L^\infty\rightarrow X$ is bounded.
\item $\frac{1}{1+s}\in X$.
\end{enumerate} 
Moreover, under any of the above assumptions, the r.i.\  optimal domain for the Hardy operator $S$ into $X$ is given by
$$
\mathfrak{D}[S,X]=\Big\{f\in L^1+L^\infty: Sf^*\in X\Big\}
$$
endowed with the norm $\|f\|_{\mathfrak{D}[S,X]}=\|Sf^*\|_X$.
\end{theorem}

\begin{proof}
The implication \textit{(i)} $\Rightarrow$ \textit{(ii)} is clear since the space $L^1\cap L^\infty$ is contained in any r.i.\ space. Also, we have \textit{(ii)} $\Rightarrow$ \textit{(iii)} trivially since  $\frac{1}{1+s}=S\chi_{[0,1]}(s)$. In order to prove the implication \textit{(iii)} $\Rightarrow$ \textit{(i)} we will actually show that the space $\mathfrak{D}[S,X]$ provides the r.i. optimal domain.

Let us see first that $\mathfrak{D}[S,X]$ is an r.i.\  space. It is clear that $\|\alpha f\|_{\mathfrak{D}[S,X]}=|\alpha|\|f\|_{\mathfrak{D}[S,X]}$ and if $f=0$, then $\|f\|_{\mathfrak{D}[S,X]}=0$. Also, if $\|f\|_{\mathfrak{D}[S,X]}=0$, then $Sf^*=0$ which means that $f=0$. Now, if $f_1,f_2\in \mathfrak{D}[S,X]$, we have that $S(f_1+f_2)^*\leq Sf_1^*+Sf_2^*$, so $\|\cdot\|_{\mathfrak{D}[S,X]}$ defines a norm and (P1) is proved.

(P2) is immediate, and (P3) follows from the fact that $f_n\uparrow f$ implies $Sf_n^*\uparrow Sf^*$, and since $X$ satisfies (P3)
$$
\|Sf_n^*\|_X\uparrow\|Sf^*\|_X.
$$

In order to check (P4), notice that by hypothesis $S\chi_{(0,1)}(s)=\frac{1}{1+s}\in X$, and since the dilation operators are bounded on $X$ for every $t>0$ we have
$$
S\chi_{(0,t)}=E_{\frac1t}S\chi_{(0,1)}\in X.
$$

Finally, (P5) holds since for every $E$ with $|E|<\infty$ we have
$$
\int_E f(s)\, ds\leq \int Sf^* \chi_{(0,|E|)}(s)\, ds\leq\|f\|_{\mathfrak{D}[S,X]}\|\chi_{(0,|E|)}\|_{X'}.
$$

Therefore, $\mathfrak{D}[S,X]$ is a well-defined r.i.\  space, and clearly for $f\in \mathfrak{D}[S,X]$ we have
$$
\|Sf\|_X\leq\|Sf^*\|_X=\|f\|_{\mathfrak{D}[S,X]}.
$$
Thus, $S:\mathfrak{D}[S,X]\rightarrow X$ is bounded. Now, suppose that $Y$ is an r.i.\  space such that $S:Y\rightarrow X$ is bounded. Then, for every $f\in Y$ we have
$$
\|f\|_{\mathfrak{D}[S,X]} = \|Sf^*\|_X \leq \|S\| \|f\|_Y,
$$
and hence, $Y\subseteq \mathfrak{D}[S,X]$.
\end{proof}

Note that the construction of $\mathfrak{R}[S,X]$ and $\mathfrak D[S,X]$ is closely related to  the optimal rearrangement invariant domains for kernel operators studied in \cite{Delgado}. In that paper, given a kernel operator of the form
$$
Tf(x)=\int_0^1f(y)K(x,y)dy,\,\,\,x\in[0,1],
$$
satisfying that for every fixed $x\in[0,1]$,   $K(x,\cdot)$ is decreasing, the author considers the space $[T,X]_{\rm r.i.}=\{f :Tf^*\in X\}$, endowed with the norm
$$
\|f\|_{[T,X]_{\rm r.i.}}=\|Tf^*\|_{X}.
$$
It turns out that under these assumptions, this space is the optimal rearrangement invariant domain for the operator $T$.

In the case of Hardy's conjugate operator $S'$, this is a kernel operator which does not satisfy the monotonicity condition required above. In fact, it is not true that $S'(f+g)^*\leq S'f^*+S'g^*$ in general (take, for example, $f=\chi_{(0,\frac12)}$ and $g=\chi_{(\frac12,1)}$). However, for any functions $f$ and $g$, we have
$$
SS'(f+g)^*=S'S(f+g)^*\leq S'(Sf^*+Sg^*)=S(S'f^*+S'g^*).
$$
Thus, using \cite[Theorem II.4.6]{Bennett-Sharpley}, for any r.i.\  space $X$ it follows that
$$\|S'(f+g)^*\|_X\leq\|S'f^*\|_X+\|S'g^*\|_X,
$$
and hence  the expression $\|\cdot\|_{[S',X]_{\rm r.i.}}$ also defines a norm in this case. Moreover,  we have the following identification:
$$
[S',X]_{\rm r.i.}=(\mathfrak{R}[S,X'])'.
$$

For example, if $X=L^1+L^\infty$, using Example~\ref{exuii} we obtain that
$$
[S',L^1+L^\infty]_{\rm r.i.}=(\mathfrak{R}[S,L^1\cap L^\infty])'=(M_{\frac{t}{\log(1+t)}})'=\Lambda_{\log(1+t)}.
$$

\begin{proposition}
Let $X$ be an r.i. space such that $S^2:L^1\cap L^\infty\rightarrow X$ is bounded. We have that the r.i. optimal domain for $S^2$ into $X$ satisfies
$$
\mathfrak D[S^2,X]=\mathfrak D[S,\mathfrak D[S,X]].
$$
\end{proposition}

\begin{proof}
Since $S:\mathfrak D[S,X]\rightarrow X$ and $S:\mathfrak D[S,\mathfrak D[S,X]]\rightarrow \mathfrak D[S,X]$ are bounded, we have that
$$
S^2:\mathfrak D[S,\mathfrak D[S,X]]\rightarrow X
$$ 
is bounded. Hence, $\mathfrak D[S^2,X]\subset\mathfrak D[S,\mathfrak D[S,X]]$. For the converse embedding, let us see that $S:\mathfrak D[S^2,X]\rightarrow \mathfrak D[S,X]$ is bounded. Given $f\in \mathfrak D[S^2,X]$, we have
$$
\|Sf\|_{\mathfrak D[S,X]}\leq\|Sf^*\|_{\mathfrak D[S,X]}=\|S(Sf^*)\|_X\leq\|S^2\|\|f\|_{\mathfrak D[S^2,X]},
$$
where $\|S^2\|$ is the norm of $S^2:\mathfrak D[S^2,X]\rightarrow X$. Thus, we also get the embedding $\mathfrak D[S,\mathfrak D[S,X]]\subset\mathfrak D[S^2,X]$.
\end{proof}

\begin{example}
It can be seen that for the space $M_{\frac{t}{\log^2(1+\sqrt{t})}}$, we have
$$
\mathfrak D[S^2,M_{\frac{t}{\log^2(1+\sqrt{t})}}]=\mathfrak D[S,\mathfrak D[S,M_{\frac{t}{\log^2(1+\sqrt{t})}}]]=\mathfrak D[S,M_{\frac{t}{\log(1+t)}}]=L^1\cap L^\infty.
$$
\end{example}

We will elaborate on the relation between r.i\ optimal domain and r.i\ optimal range by studying the following constructions.

\begin{definition}
Given an r.i\ space $X$ with $\log^+(\frac1s)\in X'$, let $\mathcal D_X$ denote the r.i.\  optimal domain for the Hardy operator into $\mathfrak{R}[S,X]$, i.e:
$$
\mathcal D_X=\mathfrak{D}[S,\mathfrak{R}[S,X]].
$$
Similarly, given an r.i.\ space $X$ with $\frac{1}{1+s}\in X$, let 
$$
\mathcal R_X=\mathfrak{R}[S,\mathfrak{D}[S,X]].
$$
\end{definition}

Note that the hypotheses on the space $X$ are necessary for the existence of the corresponding $\mathfrak{R}[S,X]$ and $\mathfrak{D}[S,X]$. The spaces $\mathcal D_X$ and $\mathcal R_X$ satisfy the following properties:

\begin{proposition} Suppose the space $X$ satisfies the conditions in the definition of $\mathcal D_X$ and $\mathcal R_X$. Then we have:
\begin{enumerate}[{(i)}]
\item $\mathcal R_X\subset X\subset \mathcal D_X$. 
\item If $\overline{\alpha}_X<1$, then $\mathcal R_X=X=\mathcal D_X$.
\item If $X\subset Y$, then $\mathcal D_X\subset \mathcal D_Y$ and $\mathcal R_X\subset \mathcal R_Y$.
\item $\mathfrak R[S,\mathcal D_X]=\mathfrak R[S,X]$ and $\mathfrak D[S,\mathcal R_X]=\mathfrak D[S,X]$. 
\item $\mathcal R_{\mathcal{R}_X}=\mathcal R_X$ and $\mathcal D_{\mathcal{D}_X}=\mathcal D_X$.
\end{enumerate}
\end{proposition}

\begin{proof}
\textit{(i)} For any decreasing function $h$ we have
$$
\|S'h\|_{\mathfrak{D}[S,X]'} = \sup_{g}\frac{\int S'h(s) g^*(s)\,ds}{\|Sg^*\|_X} = \sup_{g}\frac{\int h(s) Sg^*(s)\,ds}{\|Sg^*\|_X} \leq \|h\|_{X'}.
$$
Hence, it follows that
$$
\Vert f\Vert_{\mathcal R_X} = \sup_{h}\frac{\int_0^\infty f^*(s)h^*(s)\,ds}{\Vert S'(h^*)\Vert_{\mathfrak{D}[S,X]'}} \geq  \sup_{h}\frac{\int_0^\infty f^*(s)h^*(s)\,ds}{\|h^*\|_{X'}}  = \Vert f\Vert_X.
$$
This proves the embedding $\mathcal R_X\subset X$. For the second embedding, we have
$$
\Vert f\Vert_{\mathcal D_X}=\sup_{h}\frac{\int_0^\infty Sf^*(s)h^*(s)\,ds}{\Vert S'(h^*)\Vert_{X'}}=\sup_{h}\frac{\int_0^\infty f^*(s)S'h^*(s)\,ds}{\Vert S'(h^*)\Vert_{X'}}\le \Vert f\Vert_X.
$$

\textit{(ii)} In the particular case when $\overline{\alpha}_X<1$, which is equivalent to $S:X\rightarrow X$, we have that
$$
\mathfrak{R}[S,X]=\mathfrak{D}[S,X]=\mathcal{D}_X=\mathcal R_X=X.
$$

\textit{(iii)} If $X\subset Y$, then we have that $\mathfrak{R}[S,X]\subset \mathfrak{R}[S,Y]$ and $\mathfrak{D}[S,X]\subset \mathfrak{D}[S,Y]$. Iterating these two facts in the corresponding order yields the conclusion.

\textit{(iv)}  According to property \textit{(i)}, we have $X\subset \mathcal D_X$. So, by property \textit{(iii)}, we get that
$$
\mathfrak R[S,X]\subset \mathfrak R[S, \mathcal D_X].
$$
For the converse embedding, observe that
$$
\mathfrak R[S, \mathcal D_X]= \mathfrak R[S, \mathfrak{D}[S,\mathfrak{R}[S,X]]]=\mathcal R_{\mathfrak R[S,X]},
$$
which, again by property \textit{(i)} satisfies
$$
\mathcal R_{\mathfrak R[S,X]}\subset \mathfrak R[S,X].
$$
Thus, we have $\mathfrak R[S,\mathcal D_X]=\mathfrak R[S,X]$. The identity $\mathfrak D[S,\mathcal R_X]=\mathfrak D[S,X]$ is proved similarly.

\textit{(v)} This is a consequence of property \textit{(iv)}. Indeed, 
$$
\mathcal R_{\mathcal{R}_X}=\mathfrak R[S,\mathcal D_{\mathfrak D[S,X]}]=\mathfrak R[S,\mathfrak D[S,X]]=\mathcal R_X.
$$
And similarly,
$$
\mathcal D_{\mathcal{D}_X}=\mathfrak D[S,\mathcal R_{\mathfrak R[S,X]}]=\mathfrak{D}[S,\mathfrak{R}[S,X]]=\mathcal D_X.
$$
\end{proof}

\begin{example}
If we consider the function $\phi(t)=t\log(1+1/t)$, then
$$
\mathcal D_{\Lambda_\phi}=\Lambda_\phi.
$$
This follows from the fact that
$$
\mathfrak{R}[S,\Lambda_\phi]=L^1+L^\infty\,\, \textrm{ and }\,\, \mathfrak{D}[S,L^1+L^\infty]=\Lambda_{\phi}.
$$
Similarly, for the largest r.i.\ space $L^1+L^\infty$ we have
$$
\mathcal R_{L^1+L^\infty}=L^1+L^\infty.
$$
\end{example}

Note that this provides an equivalent condition to the existence of the r.i.\  optimal range (see Proposition~\ref{prop41}): given an r.i.\  space $X$,  the r.i.\  optimal range $\mathfrak R[S,X]$ exists if and only if $X\subset \Lambda_{\phi}$, with $\phi(t)=t\log(1+1/t)$.

In fact, the space $\mathfrak R[S,X]$ exists if and only if $S:X\rightarrow L^1+L^\infty$ is bounded, and this holds if and only if $X\subset
\mathfrak{D}[S,L^1+L^\infty]=\Lambda_\phi$.

This embedding also follows from Proposition~\ref{prop41}, using that $\chi_{[0,1]}\log\big(\frac{1}{\cdot}\big)\in X'$ and identifying $\Lambda_\phi$ as the Zygmund space $L\log L$ in $[0,1]$ \cite[IV.6]{Bennett-Sharpley}.

\begin{lemma}
Given $X$ such that $S:X\rightarrow L^1+L^\infty$, if $S:X\rightarrow \mathfrak{R}[S,X]$ is invertible (onto its image) then
$$X=\mathcal{D}_X.$$
\end{lemma}

\begin{proof}
As mentioned above, we always have the embedding $X\subset \mathcal{D}_X$. Now, given $f\in \mathcal{D}_X$ we have
$$
\|f\|_{\mathcal{D}_X}=\|f^*\|_{\mathfrak{D}[S,\mathfrak{R}[S,X]]}=\|Sf^*\|_{\mathfrak{R}[S,X]}.
$$
Hence, if $S:X\rightarrow \mathfrak{R}[S,X]$ is invertible, then there exists $\alpha>0$ such that
$$
\|f\|_{\mathcal{D}_X}=\|Sf^*\|_{\mathfrak{R}[S,X]}\geq\alpha\|f^*\|_X=\|f\|_X.
$$
\end{proof}

The following is  a useful criterion for the identity $\mathcal{D}_{\Lambda_\varphi}=\Lambda_\varphi$:

\begin{proposition}\label{criterioDLambda}
Let $\varphi$ be a quasi-concave function satisfying \eqref{logc}. Suppose that there exist $K\geq1$ and a decreasing function $g_\varphi$ such that for every $t>0$
$$
\frac1K \,\frac{\varphi(t)}{t} \leq SS'g_\varphi(t) \leq K \,\frac{\varphi(t)}{t}.
$$ Then $\mathcal{D}_{\Lambda_\varphi}=\Lambda_\varphi.$
\end{proposition}

\begin{proof}
Since $\varphi$ is a quasi-concave function,  then it can be represented as the integral of a non-negative, decreasing function $\phi$ on $\mathbb R^+$. We have that the norm in $\Lambda_\varphi$ can be computed as follows
$$
\|f\|_{\Lambda_\varphi}=\|f\|_\infty\varphi(0^+)+\int_0^\infty f^*(s) \phi(s) ds.
$$
Note that
$$
S(\phi)(s)=\frac1s\int_0^s\phi(r)\,dr=\frac{\varphi(s)}{s}\approx   SS'g_\varphi(s).
$$
So we have that
$$
\|S'g_\varphi\|_{M_{\varphi'}}=\sup_{s>0}\frac{\int_0^sS'g_\varphi(t)dt}{\varphi(s)}\leq K\sup_{s>0}\frac{\int_0^s\phi(t)dt}{\varphi(s)}=K\sup_{s>0}\frac{\varphi(s)-\varphi(0^+)}{\varphi(s)}\leq K.
$$

Let now $f\in\mathcal{D}_{\Lambda_\varphi}$. On the one hand, using \cite[Proposition II.3.6]{Bennett-Sharpley} we have that
\begin{align*}
\|f\|_{\mathcal{D}_{\Lambda_\varphi}}&=\|Sf^*\|_{\mathfrak{R}[S,\Lambda_\varphi]}=\sup_{g\downarrow}\frac{\int_0^\infty f^*(s) S'g(s)\,ds}{\|S'g\|_{M_{\varphi'}}}\\
&\geq \frac{\int_0^\infty f^*(s) S'g_\varphi(s)\,ds}{\|S'g_\varphi\|_{M_{\varphi'}}}\geq \frac1K^2 \int_0^\infty f^*(s) \phi(s)\,ds.\\
\end{align*}

On the other hand, if $\varphi(0^+)\neq0$, then for $t>0$ we have
\begin{align*}
\|S'\chi_{(0,t)}\|_{M_{\varphi'}}&=\sup_{s>0}\frac1{\varphi(s)}\int_0^s\chi_{(0,t)}(r)\log\Big(\frac{t}{r}\Big)\,dr\\
&=\max\{\sup_{s<t}\frac1{\varphi(s)}\int_0^s\log\Big(\frac{t}{r}\Big)\,dr, \sup_{s>t}\frac1{\varphi(s)}\int_0^t\log\Big(\frac{t}{r}\Big)\,dr\}\\
&\leq\frac{t}{\varphi(0^+)}.
\end{align*}
Hence, we also have
\begin{align*}
\|f\|_{\mathcal{D}_{\Lambda_\varphi}}&\geq\sup_{t>0}\frac{\int_0^\infty f^*(s) S'\chi_{(0,t)}(s)\,ds}{\|S'\chi_{(0,t)}\|_{M_{\varphi'}}}\geq \sup_{t>0}\frac{\int_0^t Sf^*(s)\,ds}{t/\varphi(0^+)} \\
&\geq \varphi(0^+)\sup_{t>0}\frac{\int_0^t f^*(s)\,ds}{t}\geq\varphi(0^+)\|f\|_\infty.
\end{align*}

Therefore, summing the two estimates we get
$$
\|f\|_{\mathcal{D}_{\Lambda_\varphi}}\gtrsim \varphi(0^+)\|f\|_\infty+\int_0^\infty f^*(s) \phi(s)\,ds=\|f\|_{\Lambda_\varphi}.
$$

Since the converse inequality
$$
\|f\|_{\mathcal{D}_{\Lambda_\varphi}}\leq\|f\|_{\Lambda_\varphi}
$$
always holds, the proof is finished.
\end{proof}

\begin{example}\label{DLambda fi alpha}
Let $\varphi_\alpha(t)=t\log^\alpha(1+t^{-1/\alpha})$, for $\alpha\geq1$. We claim that $\mathcal D_{\Lambda_{\varphi_{\alpha+1}}}=\Lambda_{\varphi_{\alpha+1}}$.

Indeed, we will show that 
$$
\frac{\varphi_{\alpha+1}(t)}{t}\approx SS'\varphi_\alpha (t),
$$
and the conclusion follows by Proposition~\ref{criterioDLambda}. In fact, in Example~\ref{rango fi alfa} we showed that for $t<1$, it holds that $SS'\varphi_\alpha(t)\lesssim\frac{\varphi_{\alpha+1}(t)}{t}$, while for $t>1$ we have $\frac{\varphi_{\alpha+1}(t)}{t}\approx SS'\varphi_\alpha (t)$. Thus, it remains to show that $SS'\varphi_\alpha(t)\gtrsim\frac{\varphi_{\alpha+1}(t)}{t}$ also holds for $t<1$. 

To this end, note that for $0<s<1$, we have $s^{\frac1\alpha}\leq s^{\frac{1}{\alpha+1}}$, and we get

\begin{equation*}
\frac1s\log^\alpha \Big(1+\frac{1}{s^{1/\alpha}}\Big)\geq\frac1s\frac{1}{(1+s^{1/(\alpha+1)})}\log^\alpha \Big(1+\frac{1}{s^{1/(\alpha+1)}}\Big)
\end{equation*}

Now, for $t<1$ we have 
\begin{align*}
SS'\varphi_\alpha(t)&=\int_t^\infty \frac{\varphi_\alpha(s)}{s^2}\,ds=\int_t^\infty\frac1s\log^\alpha\Big(1+\frac1{s^{1/\alpha}}\Big)\,ds\\
&\approx 1+ \int_t^1\frac1s\log^\alpha\Big(1+\frac1{s^{1/\alpha}}\Big)\,ds \\
&\geq 1+ \int_t^1 \frac1s\frac{1}{(1+s^{1/(\alpha+1)})}\log^\alpha (1+\frac{1}{s^{1/(\alpha+1)}})\,ds\geq\frac{\varphi_{\alpha+1}(t)}{t},
\end{align*}
where we used that $\varphi_{\alpha+1}(t)/t$ is a primitive of the last integrand.

Observe that, by the computations in Example~\ref{rango fi alfa} we have  that $\mathfrak R[S,\Lambda_{\varphi_{\alpha+1}}]=\Lambda_{\varphi_\alpha}$. Now, we have seen that $\mathcal D_{\Lambda_{\varphi_{\alpha+1}}}=\Lambda_{\varphi_{\alpha+1}}$, so we also obtain that
$$
\mathfrak D[S,\Lambda_{\varphi_\alpha}]=\mathfrak D[S,\mathfrak R[S,\Lambda_{\varphi_{\alpha+1}}]]=\mathcal D_{\Lambda_{\varphi_{\alpha+1}}}=\Lambda_{\varphi_{\alpha+1}}.
$$
\end{example}

The situation exhibited in Example~\ref{DLambda fi alpha} can be somehow extrapolated. Observe that in that case we have that $\widetilde{\varphi_{\alpha+1}}=\varphi_\alpha$, and the following identities hold
$$
\mathfrak R[S,\Lambda_{\varphi_{\alpha+1}}]=\Lambda_{\widetilde{\varphi_{\alpha+1}}}\quad\text{and}\quad\mathfrak D[S,\Lambda_{\widetilde{\varphi_{\alpha+1}}}]=\Lambda_{\varphi_{\alpha+1}}.
$$

This means that the r.i. optimal range for the Hardy operator on certain Lorentz spaces is again a Lorentz space, and the same happens with the r.i. optimal domain. We already know, by Proposition~\ref{optimal range lorentz lorentz}, when $\mathfrak R[S,\Lambda_\varphi]=\Lambda_{\widetilde{\varphi}}$. Let us study the analogous result for the optimal domain.

\begin{proposition}\label{Lorentz-Lorentz}
Let $\varphi$ be a quasi-concave function satisfying \eqref{logc}. The following are equivalent:
\begin{enumerate}[(i)]
\item $\mathfrak R[S,\Lambda_\varphi]=\Lambda_{\widetilde{\varphi}}$ and $\mathfrak D[S,\Lambda_{\widetilde{\varphi}}]=\Lambda_\varphi$.
\item For $t>0$,
$$
\int_t^\infty\frac{\widetilde{\varphi}(s)}{s^2}\,ds\approx\frac{\varphi(t)}{t}.
$$
\end{enumerate}
\end{proposition}

\begin{proof}
By Proposition~\ref{optimal range lorentz lorentz}, we already have that $\mathfrak R[S,\Lambda_\varphi]=\Lambda_{\widetilde{\varphi}}$ is equivalent to 
$$
\int_t^\infty\frac{\widetilde{\varphi}(s)}{s^2}\,ds\lesssim\frac{\varphi(t)}{t}.
$$
To finish the proof, we can assume that $S:\Lambda_\varphi\rightarrow \Lambda_{\widetilde{\varphi}}$ is bounded. Using this, now we have that $\mathfrak D[S,\Lambda_{\widetilde{\varphi}}]=\Lambda_\varphi$ holds if and only if $\mathfrak D[S,\Lambda_{\widetilde{\varphi}}]\subset\Lambda_\varphi$. This is equivalent to the estimate
$$
\int_0^\infty f^{**}(s) d\widetilde\varphi(s) \gtrsim \int_0^\infty f^*(s) d\varphi(s),
$$
and using \cite[Theorem 5.1]{CPSS} this holds if and only if
$$
\int_t^\infty \frac{\widetilde{\varphi}(s)}{s^2}\,ds\gtrsim \frac{\varphi(t)}{t}.
$$
\end{proof}

In particular, if a quasi-concave function $\varphi$ satisfies condition $(ii)$ in Proposition~\ref{Lorentz-Lorentz}, then $\mathcal D_{\Lambda_\varphi}=\Lambda_\varphi$. However, the converse is not true:

\begin{example}
Let us consider the function $\varphi(t)=\max\{1,t\}$. It is clear that $\Lambda_\varphi=L^1\cap L^\infty$. Now, it is easy to check that 
$$
\widetilde{\varphi}(t)=\frac{t}{\log(1+t)},
$$
$\mathfrak R[S,\Lambda_\varphi]=M_{\widetilde{\varphi}}$ and $\mathfrak D[S,M_{\widetilde{\varphi}}]=\Lambda_\varphi$. Thus, we have $\mathcal D_{\Lambda_\varphi}=\Lambda_\varphi$, while
$$
\int_t^\infty\frac{\widetilde{\varphi}(s)}{s^2}\,ds=\int_t^\infty\frac{1}{s\log(1+s)}\,ds=\infty.
$$

\end{example}

\section{Revisiting the spaces $R(X)$}\label{sec5}

Let us study the connection between the space $\mathfrak{D}[S,X]$ and the restricted type spaces $R(X)$. These spaces were introduced in \cite{Soria:10} in connection with the best constant for the Hardy operator minus the identity on decreasing functions and their properties have been studied in \cite{Boza-Soria,Rodriguez-Soria,Soria-Tradacete}.

By definition, the space $R(X)$ is the Lorentz space associated to the function
$$
W_X(t)=\big\|\frac{1}{1+(\frac{\cdot}{t})}\Big\|_X.
$$
Now, it is easy to see that this function is actually the fundamental function of the space $\mathfrak{D}[S,X]$:
$$
\varphi_{\mathfrak{D}[S,X]}(t)=\|S\chi_{(0,t)}\|_X=W_X(t).
$$
Hence, we have the identity
$$
R(X)=\Lambda(\mathfrak{D}[S,X]).
$$

 It was studied in \cite{Soria-Tradacete} whether every quasi-concave function satisfying that $\varphi(t)\geq C t\log(1+1/t)$ has the property that there exists an r.i. space $X$ with $R(X)=\Lambda_\varphi$. Although it is not known whether $\mathcal{D}_X=X$, for every r.i.\ space $X$,  a weaker version of this identity is actually equivalent to the previous question. 

\begin{theorem}\label{caracR(X)}
Let $\varphi$ be a quasi-concave function such that $\varphi(t)\geq Ct\log(1+1/t)$. The following are equivalent: 
\begin{enumerate}[{(i)}]
\item There exists an r.i. space $X$ such that $\Lambda_\varphi=R(X).$
\item $\Lambda_\varphi=R(\mathfrak R[S,\Lambda_\varphi])$.
\item The fundamental function of $\mathcal{D}_{\Lambda_\varphi}$ is equivalent to $\varphi$.
\item There is $K>0$ such that for every $t>0$, there exists a decreasing function $g_t$ satisfying
		\begin{enumerate}
		\item $SS'g_t(s)\leq K\varphi(s)/s$ for $s>0$; and
		\item $SS'g_t(t)\geq \frac1K \varphi(t)/t$.
		\end{enumerate}
\end{enumerate}
\end{theorem}

\begin{proof}
Recall first that the boundedness of $S:\Lambda_\varphi\rightarrow L^1+L^\infty$ is equivalent to the existence of the optimal range $\mathfrak{R}[S,\Lambda_\varphi]$, and also equivalent to the inequality $\varphi(t)\geq Ct\log(1+1/t)$.

The equivalence \textit{(i)} $\Leftrightarrow$ \textit{(ii)}  has already been established in \cite{Soria-Tradacete}.

The equivalence \textit{(ii)} $\Leftrightarrow$ \textit{(iii)} is a consequence of  the following chain of identities:
$$
R(\mathfrak{R}[S,\Lambda_\varphi])=\Lambda(\mathfrak{D}[S,\mathfrak{R}[S,\Lambda_\varphi]])=\Lambda(\mathcal{D}_{\Lambda_\varphi}).
$$
Hence, if $\Lambda_\varphi=R(X)$, then we also have that $\Lambda_\varphi=R(\mathfrak{R}[S,\Lambda_\varphi])$, and by the above identity, we get that $\varphi\approx \varphi_{\mathcal{D}_{\Lambda_\varphi}}$. Conversely, if $\varphi\approx \varphi_{\mathcal{D}_{\Lambda_\varphi}}$, then $\Lambda_\varphi=\Lambda(\mathcal{D}_{\Lambda_\varphi})=R(\mathfrak{R}[S,\Lambda_\varphi])$.

Finally, for the equivalence \textit{(iii)} $\Leftrightarrow$ \textit{(iv)}, note that
$$
\varphi_{\mathcal{D}_{\Lambda_\varphi}}(t)=\|\chi_{[0,t]}\|_{\mathcal{D}_{\Lambda_\varphi}}=\sup_{g\downarrow}\frac{\int_0^t S'g(s)\,ds}{\|S'g\|_{M_{\varphi'}}}.
$$
The inequality $\varphi_{\mathcal{D}_{\Lambda_\varphi}}(t)\leq\varphi(t)$ always holds. Therefore, condition \textit{(iii)}  is equivalent to the existence of a constant $K>0$ such that $\varphi \leq K \varphi_{\mathcal{D}_{\Lambda_\varphi}}$. This holds if and only if for every $t>0$, there is a decreasing function $g_t$ such that 
$$
\varphi(t)\leq K\frac{\int_0^t S'g_t(s)\,ds}{\|S'g_t\|_{M_{\varphi'}}}.
$$
By scaling, we can assume that $\|S'g_t\|_{M_{\varphi'}}=K^{1/2}$ (or equivalently   $SS'g_t(s)\leq K^{1/2}\varphi(s)/s$, for $s>0$), so we must have $SS'g_t(t)\geq K^{-1/2}\varphi(t)/t$.
\end{proof}

\begin{corollary}
Suppose  there exists a decreasing function $g_\varphi$ such that 
$$
\frac{\varphi(t)}{t}\approx SS'g_\varphi(t),
$$ 
then $R(\mathfrak{R}[S,\Lambda_\varphi])=\Lambda_\varphi.$
\end{corollary}

\begin{proof}
This is a direct consequence of Proposition~\ref{criterioDLambda} and Theorem \ref{caracR(X)}.
\end{proof}

\end{document}